\documentclass[11pt]{amsart}

\usepackage{txfonts}

\usepackage{amsmath}
\usepackage{amsfonts}
\usepackage{amssymb}
\usepackage{mathrsfs}
\usepackage{latexsym}
\usepackage[dvips]{graphicx}
\usepackage{xcolor}
\usepackage{dsfont}

\newtheorem{Thm}{Theorem}[section]
\newtheorem{Rmk}[Thm]{Remark}

\newtheorem{Def}[Thm]{Definition}

\def\bR {\mathbb{R}}

\def\dx {{\rm d}x}

\def\dt {{\rm d}t}

\def\R {\mathbb{R}}

\usepackage{lineno}

\begin{document}

\title{A kinetic theory approach to consensus formation in financial markets}

\bibliographystyle{abbrv}

\author[J.-G. Attali]{Jean-Gabriel Attali}
\author[F. Salvarani]{Francesco Salvarani}
\address{J.-G.A.: De Vinci Higher Education, De Vinci Research Center, Paris, France}
\email{jean-gabriel.attali@devinci.fr}
\address{F.S.: De Vinci Higher Education, De Vinci Research Center, Paris, France \& Dipartimento di Matematica ``F.
Casorati'', Universit\`a degli Studi di Pavia, Italy}\email{francesco.salvarani@unipv.it}

\begin{abstract}
It is sometimes acknowledged that (sell-side) equity analysts' recommendations influence investors and therefore market prices. In particular, the S\&P 500 is expected to decline (respectively rise) when analysts revise their targets downward (respectively upward). Our findings indicate not only that analysts' consensus exert no influence on market prices, but also that, conversely, analysts appear to set their target prices based on markets prices.

 Employing a kinetic theory framework, we model the dynamics of analysts' opinions, by taking into account both the mutual influences shaping price consensus and the dynamics of the actual S\&P 500 index level.
The model is calibrated on a training subset of data and tested on an independent set to assess its predictive power.
Our tests show that just three free parameters are enough to accurately predict the one-year average price forecasts of analysts.
\end{abstract}

\keywords{ S\&P 500; Cointegration; Kinetic modeling; Forecast analysis. 
\textit{JEL Codes:} C63}

\maketitle

\section{Introduction} \label{intro}

Opinion dynamics play an important role in understanding the emergence of collective behaviors in markets, particularly under external influences such as price movements. Conventional perspectives often link price evolution to shifts in 
analysts' recommendations.  In contrast, our work investigates the reverse relationship: how 
analysts' recommendations 
 react to and closely follows price dynamics. This approach is essential for unraveling the feedback loops that drive decision-making within financial systems.

The application of methods from statistical mechanics to model social and economic systems has a long and successful history. 
This interdisciplinary approach was pioneered by Weidlich \cite{wei} and further popularized by Galam, Gefen, and Shapir \cite{gal-gef-sha}. 
By treating individuals as interacting entities and modeling their collective behavior, researchers have gained insights into complex social phenomena. 
 
Indeed, kinetic theory,
originally developed in statistical physics to describe the evolution of particle distributions in a gas, provides a natural framework for modeling the consensus formation process in an interacting community \cite{pareschi2013interacting}, \cite{albi2017recent}.

Kinetic models typically describe the time evolution of a probability distribution, capturing the stochastic nature of agent interactions and the emergence of macroscopic patterns from microscopic rules (see, for example, \cite{MR3579562}, \cite{bou-sal}, {\cite{BSS0},}
\cite{BSS}, \cite{hel93b}, \cite{hel93a}, \cite{hel94}, \cite{tos06}).

This paper investigates consensus formation within the field of financial analysis. Employed by financial institutions such as investment banks or brokerage firms, sell-side equity analysts are tasked with evaluating the financial performance of publicly traded companies. Their primary objective is to formulate investment recommendations for equity investors. These recommendations typically take the form of 'buy,' 'hold' (or 'neutral'), or 'sell' ratings, often accompanied by a one-year price target for the covered securities. It is crucial to note that equity analysts are not investors themselves; indeed, they are explicitly prohibited from trading in the shares of the companies they cover. Their market influence, therefore, is derived exclusively from the market's reaction to the dissemination of their recommendations, given that a substantial portion of investors act upon these advisories. 

In the context of financial forecasting, analysts continuously update their expectations based, for example, on new information, market trends, and (possibly) peer influence. This process can be viewed as an interaction-driven system, where opinions evolve through repeated exchanges and adjustments. Unlike deterministic approaches that 
{produce a single predicted}
value, kinetic models allow for a richer description by capturing the full distribution of expectations, thus accounting for heterogeneity in analysts' predictions. 

For instance, it allows us to compute higher-order moments such as the variance, as detailed in Section \ref{s:kin}, which are crucial for understanding uncertainty and predicting risk or volatility.

Furthermore, such models can describe key phenomena such as polarization, clustering of forecasts, and the convergence toward consensus.  

A kinetic approach is particularly well suited for studying the expected level of the S\&P 500 index at a one-year horizon, as it allows us to incorporate both endogenous factors, such as analysts' interactions, and exogenous influences, including macroeconomic shocks and financial news. By extending methods from kinetic theory to this setting, we aim to provide a realistic and flexible model for predicting 
analysts' recommendations,  ultimately improving our understanding of how collective financial expectations evolve over time.

Specifically, we focus on the one-year consensus forecasts of analysts aggregated across all stocks belonging in the S$\&$P 500 index. We first demonstrate that the S$\&$P 500  and its consensus target price, more precisely, the logarithms of these two series, are cointegrated. This property is easily confirmed by using Engle-Granger procedure \cite{engle}. The short-run dynamics is also investigated  thanks to the 

Johansen-Juselius multivariate cointegration technique \cite{jorank}, \cite{jo} and \cite{juselius}.

While it is not unexpected that these two series share a common trend, we find, surprisingly, that Granger causality test leads to the conclusion that the S$\&$P 500 level causes the consensus target price, but not vice versa. This result appears to be contradictory with several studies on the impact of analyst revisions on prices (see for instance \cite{brav}, \cite{park}).  It is important to note, however, that these authors focus on analyst earnings revisions, while we examine their ability to forecast prices at a one-year horizon. Furthermore, these studies focus on individual stocks, while we are interested in the aggregate forecast for the entire S$\&$P 500. 

Consequently, a central hypothesis of our study, while designing the kinetic model, is that the 
{analysts'}
consensus
does not predict or influence market movements, but rather follows them. Our model assumes that two main factors drive the analysts' opinions: the influence of other analysts (described as a mean-field interaction) and the actual price of the index. This framework allows us to explore the interplay between these two effects and their impact on opinion formation. 

Furthermore, our analysis shows that the influence of the mean-field becomes negligible over time, leaving the current price as the sole determinant of the average one-year forecasts. 
Hence, the average of the analysts' one-year forecasts on the S\&P 500 index can be {modeled} using only three fixed parameters: $\Delta$, representing the {relative premium applied as a proportional adjustment to the current index level}; $q$, a 
{convex combination} coefficient 
{weighting the} current index value $X(t)$,
{augmented by a spread,} and the analyst's 
{prior forecast}; and $\beta$, 
{modulating the relative influence of the index dynamics with respect to the mean-field consensus.}

{To validate the robustness of our approach, we divided the dataset into two parts: a training set, used to estimate the model's parameters, and a testing set, used to evaluate its predictive performance on data not seen during calibration.}

By demonstrating the reactive nature of the consensus forecast and quantifying their dependence on price movements, this work contributes to a deeper understanding of opinion dynamics in financial markets.

\section{Exploratory data analysis based on Engle and Granger approach}
\label{s:data}

\subsection{Data}

Our study uses the one-year target prices for S$\&$P 500 stocks published by Bloomberg. These targets are provided by equity analysts at banks and brokerages. 

Bloomberg aggregates this data and distributes it to market participants through its terminal. The company then calculates a one-year consensus forecast for each stock, using the average of all individual target prices. Bloomberg finally computes the sum of one-year consensus for S\&P 500 stocks weighted by the weight of each stock in the index to obtain a one-year consensus forecast for the level of the S\&P 500 based on analysts' recommendations. 
 Equity analysts are typically required to publish reports on a regular basis, at least once per quarter. Nevertheless, they also provide commentary whenever events are expected to influence the price of the stocks they cover. As a result, an analyst may revise his target price several times within a year. More than 30 banks and brokerage firms contribute to Bloomberg, generally the biggest ones. Given that the S$\&$P 500 is composed of 500 stocks, the aggregate target price is subject to daily fluctuations, although the magnitude of these changes may be small.
 The daily data analyzed in this article cover the period from May 26th, 2009, to  June 2nd, 2025.

\begin{figure}
\begin{center}
\includegraphics[height=9cm]{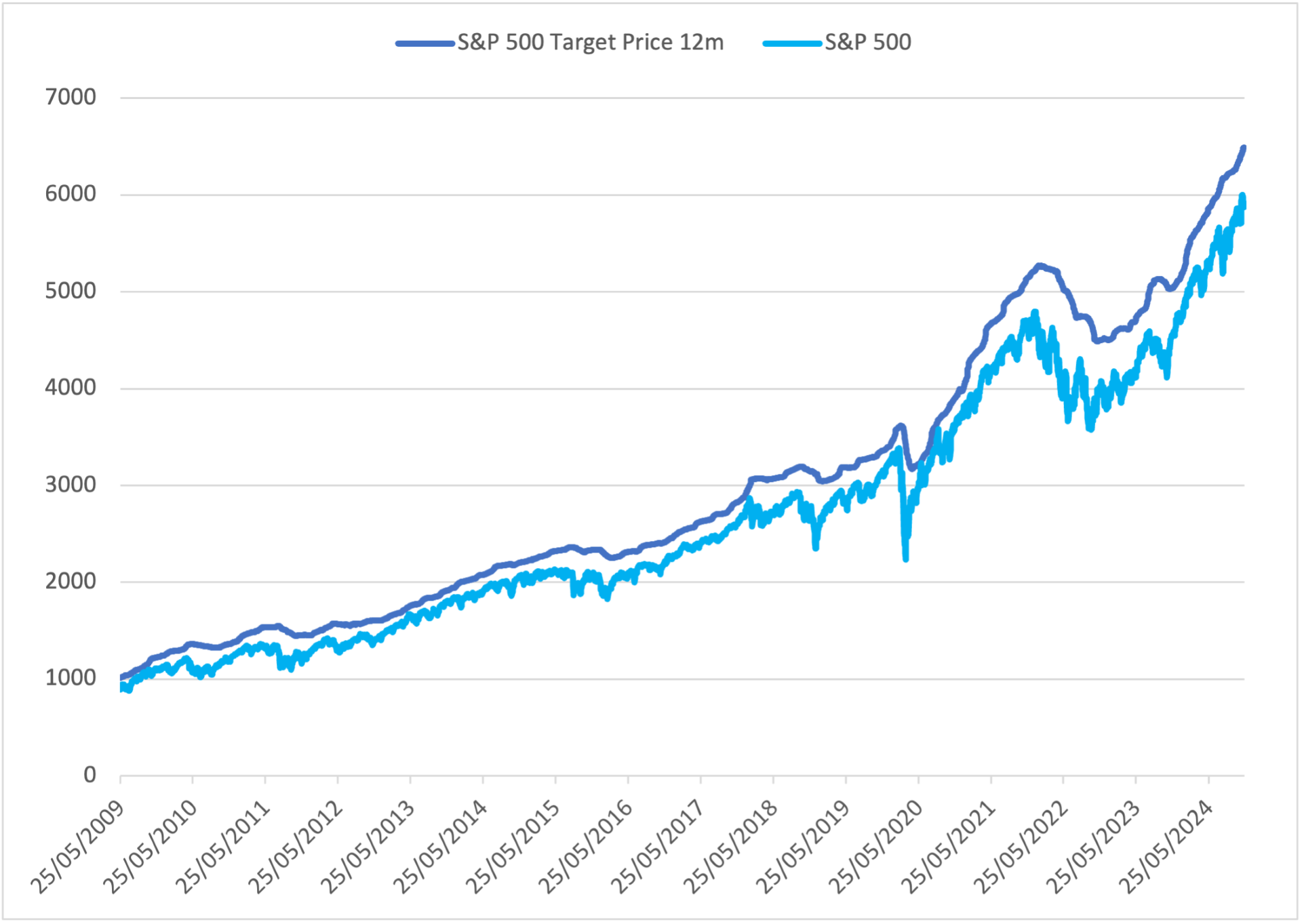}
\caption{
Time series of the S\&P 500 index $X(t)$ 
{(dark grey line)} and the one-year analysts' forecasts for the S\&P 500 index
{(grey line)}, from May 26th, 2009, to  {June 2nd, 2025}.
Source: Bloomberg} 
\label{fig:01}
\end{center}
\end{figure}

In the remainder of this paper, $X(t)$ will denote the level of the S\&P 500 at time $t$ while $s_{\text{meas}}(t)$ will be the one year consensus target price computed by Bloomberg for the S\&P 500. 

\subsection{Test for stationarity}
 
It is quite obvious that $\ln (X(t))$ and $\ln (s_{\text{meas}}(t))$ are nonstationary and, indeed, stationarity is rejected by the Augmented Dickey-Fuller test at any confidence level (cf table \ref{tableadf1}).

\begin{table}
\begin{center}
\caption{Augmented Dickey-Fuller test for $\ln (X(t))$ and $\ln (s_{\text{meas}}(t))$}
\label{tableadf1}
\begin{tabular}{lccc} \hline
 &  & t-statistic & Prob \\ \hline
ADF test statistics for $\ln (X(t))$ &  & 0.74 & 0.34 \\
ADF test statistics for $\ln (s_{\text{meas}}(t))$ &  & 2.35 & 0.62 \\
\hline
Test critical values & 1\% level &  -2.58 \\
Test critical values & 5\% level &  -1.95 \\
Test critical values & 10\% level &  -1.62 \\
\hline
\end{tabular}
\end{center}
\end{table}

For a discrete time series $Y(t)$, define the first difference $DY(t)=Y(t)-Y(t-1)$. This time, $D\ln (X(t))$ and $D\ln (s_{\text{meas}}(t))$ are accepted to be stationary according to the ADF test (cf table \ref{tableadf2}).

\begin{table}
\begin{center}
\caption{Augmented Dickey-Fuller test for $D\ln (X(t))$ and $\ln (s_{\text{meas}}(t))$}
\label{tableadf2}
\begin{tabular}{lccc} \hline
 &  & t-statistic & Prob \\ \hline
 ADF test statistics for $D\ln (X(t))$  &  &  -8.26  & $<$0.01    \\ 
 ADF test statistics for $D\ln (s_{\text{meas}}(t))$  &  &  -4.62  &  $<$ 0.01      \\ 
\hline
Test critical values & 1\% level &  -2.58  &     \\ 
Test critical values & 5\% level &  -1.95  &     \\ 
Test critical values & 10\% level &  -1.62  &     \\ 
\hline
\end{tabular}
\end{center}
\end{table}

Finally, the time series $\ln (X(t))$ and $\ln (s_{\text{meas}}(t)$ are accepted as $I(1)$ (integrated of order 1) because they are non-stationary series whose first differences are stationary (or $I(0)$).

\subsection{Cointegration test}

Now, we investigate the cointegration between the quantities $\ln (X(t))$ and $\ln (s_{\text{meas}}(t))$, which is a linear relation linking theses series. In that particular case, the cointegration relationship would be defined as follows:

\begin{Def}[]\label{Cointegration}
 Suppose discrete time series $Y(t)$ and $Z(t)$ are $I(1)$. Then $Y(t)$ and $Z(t)$ are said to be cointegrated (of order 1) if there exist $\alpha$, $\beta$ and a time series $u(t)$ which is $I(0)$ such that:
\begin{equation}\label{lt} Y(t)=\alpha Z(t) +\beta +u(t).\end{equation}
\end{Def}

\begin{Rmk} Concepts of integration and cointegration can easily be generalized to higher order but we don't need this general framework here.
\end{Rmk}

Usually, one cannot use the Ordinary Least Square method to estimate parameters in the  cointegration relationship
(\ref{lt}) for it leads to false conclusions. This issue has first been exhibited by Yule (see \cite{yule}). Decades later, Granger and Newbold investigated this misprocedure more deeply \cite{granger} and baptised it ``spurious regression''. Thereafter, Engle $\&$ Granger (\cite{engle}) proposed the following procedure to investigate cointegration and estimate the parameters whenever $Y(t)$ and $Z(t)$ are proved to be $I(1)$:

\begin{itemize}
\item[-] First, estimate the (long-run) relationship (\ref{lt}) between $Y$ and $Z$ by computing OLS for the linear regression of $Y$ on $Z$ (even though it is not allowed). 
\item[-] Then compute $e(t)=Y(t)-\widehat{\alpha}Z(t) -\widehat{\beta}$.
\item[-] At last, test stationarity for $e(t)$ by using Augmented-Dickey Fuller test. If stationarity is accepted, then $Y(t)$ and $Z(t)$ are cointegrated and $\widehat{\alpha}$ and $\widehat{\beta}$ are  convergent estimators for $\alpha$ and $\beta$.

\end{itemize}

The estimation of the linear regression $\ln (s_{\text{meas}}(t)) = \alpha\ln (X(t))+
\beta +u(t)$ leads to $\widehat{\alpha}=0.98679$ and $\widehat{\beta}=0.21778$. Augmented Dickey-Fuller test allows us to accept stationarity for $e(t)=\ln (s_{\text{meas}}(t))-0.98679\ln (X(t)) - 0.21778$ (cf table \ref{tableadf3}).

\begin{table}
\begin{center}
\caption{Augmented Dickey-Fuller test applied to the residuals from the regression  $\ln (s_{\text{meas}}(t))$ on $\ln (X(t))$}
\label{tableadf3}
\begin{tabular}{lccc} \hline
   &  & t-statistic & Prob \\ \hline
ADF test statistics for $e(t)$  &  &  -2.82  & $<$0.01       \\ 
\hline
Test critical values & 1\% level &  -2.58  &     \\ 
Test critical values & 5\% level &  -1.95  &     \\ 
Test critical values & 10\% level &  -1.62  &     \\ 
\hline
\end{tabular}
\end{center}
\end{table}

This leads to the conclusion that $\ln (X(t))$ and $\ln (s_{\text{meas}}(t))$ are cointegrated. This appears clearly by looking at  figure \ref{fig:02} which compares $\ln (s_{\text{meas}}(t))$ with its estimated long-term equilibrium  level $\alpha\ln (X(t))+
\beta$.

\begin{figure}
\begin{center}
\includegraphics[height=9cm]{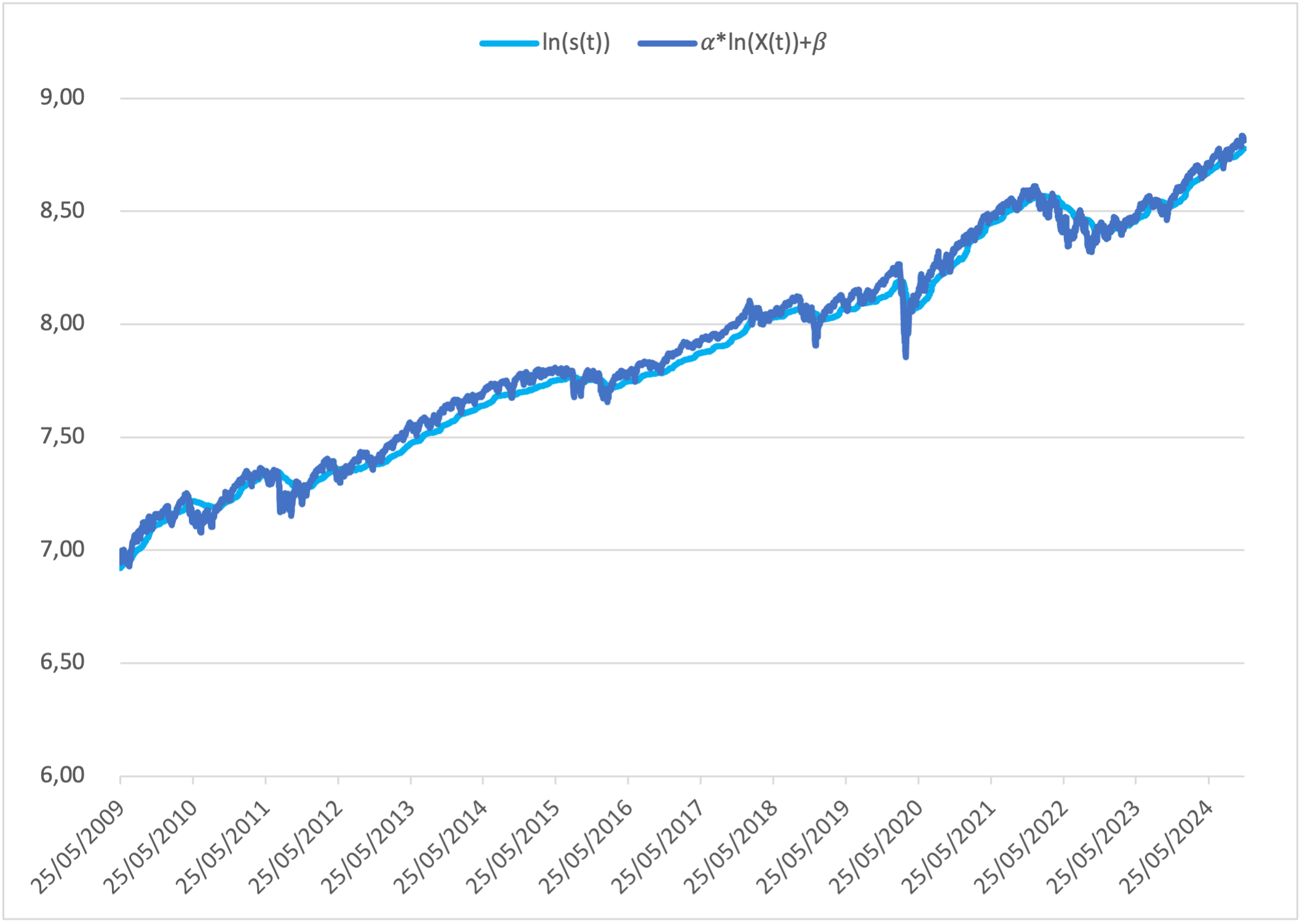}
\caption{
Logarithm of one-year analysts' forecasts for the S\&P 500 index {(dark grey line)} vs its long-term equilibrium level
{(grey line)}, from May 26th, 2009, to  {June 2nd, 2025}.} 
\label{fig:02}
\end{center}
\end{figure}

\subsection{Short-run relationship}

One major result regarding cointegration theory is the {Granger-Engle Representation Theorem} \cite{engle}:

\begin{Thm}\label{Cointegration2}
 Suppose $Y(t)$ and $Z(t)$ are cointegrated at the order 1. Then there exists  an {error-correction representation} i.e. an (error-correction) model of the form:
\begin{equation}\label{ct} D Y(t) =  \sum_{i=1}^p a_i DY(t-i) +\sum_{j=1}^p b_j DZ(t-j)+ \gamma \left(Y(t-1)-\alpha Z(t-1) -\beta\right)+\varepsilon(t),
\end{equation}
where $\varepsilon(t)$ is $I(0)$.

\end{Thm}

\begin{Rmk} 
{The name of Model (\ref{ct}) stems from the fact that the error in the cointegration relationship at time $t-1$ 
is corrected at time 
$t$ by the restoring force 
$\gamma$, which must be negative.}
The more the  ``error" at time $t-1$ in the cointegration relationship is large, the stronger the restoring force is.
\end{Rmk}

One can notice that there  also exists an error correction model using $DZ(t)$ as the dependent variable since Equation (\ref{lt}) can be written:
$$Z(t)=\frac{1}{\alpha} Y(t) -\frac{\beta}{\alpha} -\frac{1}{\alpha}u(t).$$ 
Accordingly, we need to estimate two short-term equations simultaneously. To investigate this further, we will use the Johansen procedure, as detailed in the following section.

\begin{Rmk} The long-run relationship in the same in both  short-term models, $s_{\text{meas}}(t)$ being the response variable. Consequently, $\gamma'$ estimated below will remain negative while $\gamma$ will become positive. 
\end{Rmk}

\section{Johansen procedure}

We use the Johansen procedure with Maximum Likelihood Estimation (MLE) to determine the parameters of the Vector Error Correction Model (VECM). This estimation is performed in the \texttt{R} programming environment using {\texttt {cajorls}} and {\texttt{vecm}} functions. A crucial initial step is testing the rank, $r$, of the VECM's long-run coefficient matrix $\Pi$. This rank $r$ represents the number of distinct cointegration relationships. The Johansen procedure is preferred over the Engle-Granger approach because it can simultaneously handle multiple variables and relationships. However, since our analysis is limited to only two variables,  ($\ln (s_{\text{meas}}(t))$ and $\ln (X(t))$), we can have at most one cointegration relationship $(r\leq 1)$.

\subsection{Tests for lags and cointegration}

Based on three out of four information criteria (Akaike, Schwarz, Hannan-Quin and Final Prediction error), the optimal number of lags 
$p$ for short-run dynamics is estimated to be 9 (see Table \ref{table0}). This is therefore the value that will be retained to estimate the model.

\begin{table}
\begin{center}
\caption{Optimal number of lags for each criterion} 
 \label{table0}
\begin{tabular}{cccc} \hline
     AIC & SIC & HQ & FPE \\\hline
 9 & 5 &  9  & 9     \\ 
\hline
\end{tabular}
\end{center}
\end{table}

The need for a large number of lags stems from the large residual outliers observed in some sub-periods of the cointegration system. This is a recognized characteristic of equity markets: their frequent experience of volatility regime changes. However, we confirmed that the optimal number of lags decreases with the sample size; for instance, restricting the analysis to the first three years of historical data only requires four lags, as reported in Table \ref{table3Y}.

\begin{table}
\begin{center}
\caption{Optimal number of lags for each criterion (2009-2012) } 
 \label{table3Y}
\begin{tabular}{cccc} \hline
     AIC & SIC & HQ & FPE \\\hline
 4 & 3 &  4  & 4      \\ 
\hline
\end{tabular}
\end{center}
\end{table}

If the estimation is restricted to the first year of the historical data, the necessary number of lags is lowered even further, as reported in Table \ref{table1Y}."

\begin{table}
\begin{center}
\caption{Optimal number of lags for each criterion (2009-2010) } 
 \label{table1Y}
\begin{tabular}{cccc} \hline
     AIC & SIC & HQ & FPE \\\hline
  2 & 2 &  2  & 2     \\ 
\hline
\end{tabular}
\end{center}
\end{table}

To test the number of cointegration relationship, there are two methods (trace and maximum eigenvalue), but again, these are identical if we only deal with two variables. Now, the cointegration rank test lead to the result reported in table (\ref{rank}).
\begin{table}
\begin{center}
\caption{Johansen rank test}
 \label{rank}
\begin{tabular}{lcccc} \hline
 & test  & 10\%    & 5\%  & 1\%   \\\hline
$r\leq 1$ & 11.81  &  7.52  & 9.24 & 12.97     \\ 
$r=0$ & 178.74 &  13.75  & 15.67 & 20.20    \\
\hline
\end{tabular}
\end{center}
\end{table}

Since we cannot reject the presence of a single cointegrating relationship at the 1$\%$ significance level, and since both time series were previously established as $I(1)$ variables, we have strong grounds to conclude that a long-run cointegrating relationship exists.

\subsection{Estimates}

Even though the Johansen procedure allows for the estimation of both the short-run and the long-run relationships in one step, we are going to comment on each relationship individually.

\subsubsection{Long-term relationship }

The estimation leads to:

$$\ln (s_{\text{meas}}(t))-0.99186 \ln (X(t)) -0.20574\leadsto I(0).$$

This result is very close to the estimate obtained earlier using the 
{simpler} Engle-Granger approach. It is this term that provides the mechanism to correct for errors in the short-run model.

\subsubsection{Short-run relationship with $D \ln (X(t))$ as the dependent variable}

{The equation is the following:}

$$D\ln \left(X(t)\right) =\sum_{i=1}^9 a_iD\ln (X(t-i)) +\sum_{j=1}^9 b_j D\ln (s_{\text{meas}}(t-j))+ \gamma\left(\ln (s_{\text{meas}}(t-1))-\alpha \ln (X(t-1)) -\beta\right)+\varepsilon(t).$$
None of the coefficients of the past level of $D\ln (s_{\text{meas}}(t))$ (i.e. from $b_1$ to $b_9$) are significant (see Table \ref{table1} below). Hence $\ln (s_{\text{meas}}(t))$ does not (Granger-)cause $\ln (X(t))$. The $R^2$ is only $2\%$, which is very small. We can deduce that the short-term dynamics of the market is not impacted at all by the consensus
{price}. One can remark, though, that there is a bit of autocorrelation for $\ln (X(t))$  (e.g. $a_1$ is significant).

\begin{table}
\begin{center}
\caption{Estimates of the  short-term equation for {the} response variable $D\ln (X(t))$ \newline
***, **, *, $\cdot$ indicate statistical significance at $0.1\%$, $1\%$, $5\%$, $10\%$ levels respectively}
 \label{table1}
\begin{tabular}{cccc} \hline
  Parameter & Estimate & Std & $t$-stat  \\\hline
$\gamma$ & 0.002183 &  0.005223  & 0.418     \\ 
$a_1$ &  -0.087234  & 0.015780 & -5.528***  \\ 
$a_2$  &  0.053540  & 0.015902  & 3.367*** \\ 
$a_3$ &  -0.028266 &  0.015948 &  -1.772 $\cdot$\\ 
$a_4$ &  -0.025171  & 0.015985 & -1.575    \\ 
$a_5$ &  -0.010029 &  0.016015 & -0.626   \\ 
$a_6$ &  -0.063941 &  0.016094 & -3.973***   \\ 
$a_7$  &  0.054246  & 0.016146  & 3.360***   \\ 
$a_8$ &  -0.044517  & 0.016156 &  -2.755**   \\ 
$a_9$   & 0.060541  & 0.016119 &  3.756***\\ 
$b_1$ & -0.106482 &  0.208119 & -0.512    \\  

$b_2$  & 0.190598  & 0.220288 &  0.865 \\ 

$b_3$  & 0.179068 &  0.220953 &  0.810 \\ 

$b_4$  & -0.024094 &  0.221821 & -0.109 \\ 

$b_5$  & 0.226850 &  0.221784  & 1.023 \\ 

$b_6$  & -0.218241 &  0.221501 &  -0.985 \\ 

$b_7$  & 0.131186 &  0.220381 &  0.595 \\ 

$b_8$ & -0.266791  & 0.219320  & -1.216 \\ 

$b_9$  & -0.138858 &  0.205359 & -0.676 \\ 
\hline
\end{tabular}
\end{center}
\end{table}

\subsubsection{Short-run relationship with $D\ln (s_{\text{meas}}(t))$ as the dependent variable}

{The studied equation is the following:}

$$D\ln (s_{\text{meas}}(t)) =\sum_{i=1}^9 a'_i D\ln (s_{\text{meas}}(t-i)) +\sum_{j=1}^9 b'_jD\ln(X(t-j))+ \gamma'\left(\ln (s_{\text{meas}}(t-1))-\alpha \ln (X(t-1)) -\beta\right)+\varepsilon'(t).$$

The short-term equation with $D\ln (s_{\text{meas}}(t))$ as the dependent variable is reported in  Table \ref{table2}. The $R^2$ of the model is around $60\%$, wich is quite good, and most of the coefficients are significant.

\begin{table}
\begin{center}
\caption{Estimates of the  short-term equation for  the response variable $D\ln (s_{\text{meas}}(t))$\newline
***, **, *, $\cdot$ indicate statistical significance at $0.1\%$, $1\%$, $5\%$, $10\%$ levels respectively}
 \label{table2}
\begin{tabular}{cccc} \hline
  Parameter & Estimate & Std & $t$-stat  \\\hline
$\gamma'$ & -0.0047212 &   0.0003967  & -11.901***   \\ 
$a'_1$ &  0.3476735 & 0.0158070 & 21.995 ***  \\ 
$a'_2$  &  0.0829855 & 0.0167312  & 4.960 *** \\ 
$a'_3$ &  0.0978602 & 0.0167817 &  5.831 *** \\ 
$a'_4$ &  0.0643269 & 0.0168477 &  3.818 ***    \\ 
$a'_5$ &  0.0582016 & 0.0168449 &  3.455 ***  \\ 
$a'_6$ &  0.0195366 & 0.0168233 &  1.161       \\ 
$a'_7$  &  0.0014429 & 0.0167383 &  0.086       \\ 
$a'_8$ &  -0.0271617 & 0.0166577   & -1.631       \\ 
$a'_9$   & 0.0157422 & 0.0155973  & 1.009  \\ 
$b'_1$ &  0.0074338  & 0.0011985 &  6.202 ***    \\  

$b'_2$ & 0.0063950 & 0.0012078 &  5.295 *** \\ 

$b'_3$ & 0.0066243 &  0.0012113 &  5.469 *** \\ 

$b'_4$  & 0.0050619 & 0.0012141 &  4.169 *** \\ 

$b'_5$ & 0.0080997 & 0.0012164  & 6.659 *** \\ 

$b'_6$  & 0.0045694  & 0.0012224 &  3.738 ***\\ 

$b'_7$  & 0.0043499 & 0.0012263 &  3.547 *** \\ 

$b'_8$ & 0.0053319 & 0.0012271  & 4.345 ***\\ 

$b'_9$  & 0.0061260 & 0.0012243  & 5.004 ***\\ 
\hline
\end{tabular}
\end{center}
\end{table}

This equation is the most relevant one, as it demonstrates that the consensus target price for the S$\&$P 500 adjusts based on the market level, rather than the other way around. As a result, the consensus 
{(i.e. the one-year target price)}
acts more as a ``market follower" than a ``market mover" in the short term.

\begin{Rmk}  One can remark that $\alpha$ in the long-run equation is near 1. The Augmented Dickey-Fuller test for $\ln (s_{\text{meas}}(t))-\ln (X(t))$ leads to the conclusion that one can accept that $\alpha$ is equal to 1, which, again, is no surprise when looking at Figure \ref{fig:01}.
\end{Rmk}

\begin{Rmk} The  half-life of the second short-run relationship is the number of periods that are required to correct half the deviation from equilibrium in the long-run relationship. It is computed by $t$ such that $$(1+\gamma')^t=\frac{1}{2}$$
which yields to $t=147$ trading days, i.e. between 6 and 7 months. 
\end{Rmk}

\subsection{Diagnostic tests for residuals}

To perform the diagnostic tests, we first need to convert the VECM (Vector Error Correction Model) into its CVAR (Cointegrated Vector Autoregression) form. We use the \texttt{vec2var} procedure in \texttt{R} to complete this transformation and retrieve the necessary residuals

\subsubsection{Serial Correlation}

The Breusch-Godfrey serial correlation test of the residuals (see table \ref{correlation} below) confirms the absence of autocorrelation, validating that the lag length in the model was correctly specified.

\begin{table}
\begin{center}
\caption{Breusch-Godfrey test}
 \label{correlation}
\begin{tabular}{cc} \hline
  Statistic  & p-value  \\\hline
 49.636 & 0.14     \\ 
\hline
\end{tabular}
\end{center}
\end{table}

\subsubsection{Check for normality}

In contrast, the normality assumption for residuals is rejected (see table \ref{jb}). This is a known result, at least regarding the log returns of the S$\&$P 500. Indeed, even though this assumption is implicit in most financial models (Markowitz Optimization, Black-Scholes Equation), it is generally not met for stock returns. Nevertheless, given our large sample size, the Central Limit Theorem is relied upon to ensure the validity of our inferences.

\begin{table}
\begin{center}
\caption{Jarque-Bera test}
\label{jb}
\begin{tabular}{cc} \hline
  Statistic  & p-value  \\\hline
 76932 & 0.0000      \\ 
\hline
\end{tabular}
\end{center}
\end{table}

\subsubsection{Arch effect}

Conversely, heteroscedasticity in the residuals is observed (see Table \ref{arch}). This issue is likely attributable to the volatility regime changes in equity markets, the same phenomenon that necessitated retaining a large number of lags in the model. Ideally, the equilibrium model should be estimated on a rolling basis and re-estimated regularly. However, even though this procedure improves the test statistics for ARCH effect, the homoscedasticity assumption is consistently rejected, regardless of a reduced sample size. This persistence is arguably the model's main weakness. Despite this limitation, the approach proved valuable for understanding the underlying dynamics of consensus formation, allowing us to successfully calibrate the kinetic model (see next section), which inherently requires no assumptions about the distribution of log returns or homoscedasticity.

\begin{table}
\begin{center}
\caption{Heteroscedasticity test  for residuals}
\label{arch}
\begin{tabular}{cc} \hline
  Statistic  & p-value  \\\hline
 1613 & 0.0000       \\ 
\hline
\end{tabular}
\end{center}
\end{table}

\section{The kinetic model} \label{s:kin}

The kinetic-type model described in this Section predicts the behavior of the
{analysts' recommendations on the one-year} S\&P 500 index by supposing that the influence on the
index forecasts is governed by a small number of inputs, namely the opinions of the other agents (consensus), the market suggestions, and the degree of attractiveness of the share price.

The model's independent variables are time, $t \in \mathbb{R}^+$, and the one-year price forecast for a given agent, denoted by $x \in \R_+$.

The agents' community is described by the distribution function $f = f(t, x) $, defined on $\mathbb{R}^+ \times \bar{\Omega} $. This function represents the density of agents holding price forecast $x $ at time $t $. Its evolution is governed by a kinetic equation, encompassing two levels of interaction: the microscopic level, where opinions on the price forecasts are influenced by the mean field and external market forces, and the macroscopic level, which captures the evolution of $f(t, x) $ induced by these microscopic interactions.

\subsection*{Interaction with the price dynamics}

The influence of price dynamics on the one-year price forecast evolution is governed by four primary features:

\begin{itemize}
\item[-] Strength of the consensus on the forecast. It is described by a time-dependent parameter $\alpha \in \mathbb{R}^+$ that quantifies the intensity of the influence exerted by the forecast of the other agents on the evolution of one-year price forecasts.

\item[-] The share price. It is described by a function $X(t): \mathbb{R}^+ \to \bar{\Omega}$. The perfect knowledge of the share price reflects changes in the market environment and serves as the primary external factor affecting individual opinions on price forecasts.

\item[-] Attractiveness. The attractiveness is described by a parameter $q \in (0,1)$ that quantifies the degree to which a one-year price forecast $x $ is drawn toward the share price $X(t)$ {augmented by a premium}.

\item[-] Tendency to consensus. This feature implies that opinions depend not only on the attractiveness of the share price at time $t$, but also on the mean field of price opinions.
\end{itemize}

The central mechanism of the model lies in the interaction rule, which reflects the natural tendency of opinions to converge while being influenced by price dynamics. Following \cite{BSS}, we describe the evolution of opinions as:
\begin{equation}
\label{collisions_media}
\bar{x} =  (1-q)x + q X(t)(1+\Delta),
\end{equation}
where $\Delta \in \mathbb{R}$ is {the} spread that encapsulates the deviation of forecasts from the actual market price.

We assume that the time dynamics of $X(t) $ is known. 
For technical reasons, we suppose moreover that the share price is a continuous and uniformly bounded function of time.

Inverting the interaction rule in \eqref{collisions_media}, we obtain:
$$
x = \frac{1}{1-q}\bar{x} - \frac{q}{1-q} (X(t) + \Delta).
$$

\subsection*{Time Evolution of $f $}

We first introduce the 
{
function $s \,:\, \R^*_+\to  \R_+$, which describes the average recommendation}
of the 
{analysts'} population on the target price, defined as:
\begin{equation}
\label{e:s_old}
s(t) = \frac{\displaystyle \int_{\R_+} x f(t, x) \, \mathrm{d}x}{\displaystyle  \int_{\R_+} f(t, x) \, \mathrm{d}x}.
\end{equation}
{In what follows, this quantity will sometimes be indicated as ``the sentiment''.}

The time evolution of the opinion distribution $f $ is governed by a non-local partial differential equation.
In its classical form, the model is the following:
\begin{equation}
 \label{e:model}
 \left\{
 \begin{array}{l}
\displaystyle \partial_t f + \alpha \partial_x \big((s(t) - x)f\big)=
\displaystyle \beta \left [\frac 1 {1-q} f\left (\frac x {1-q}  -\frac q{1-q} X(t)(1+\Delta)\right)\mathbf{1}_{\left (\frac x {1-q}  -\frac q{1-q} X(t)(1+\Delta)\geq 0 \right)}
 -f(x)\right]\\[10pt]
f(0, x) = f^{\text{in}}(x) \qquad x \in {\R^*_+} \\[10pt]
f(t,0)=0 \qquad t \in {\R^*_+},\\[10pt]
 \end{array}
 \right.
\end{equation}
where $\alpha\in\bR_+^*$ is a parameter which describes how strong is the comparison with 
{$s$} and $\beta\in\bR_+^*$ is the characteristic relaxation time to the spot price.

\subsection*{The {semi-}weak formulation}

It is convenient to express and to study the previous model in semi-weak form. 
Let 
{$\phi\in 
C^1(\R)$ and}
suppose that
$$
\lim_{x\to 0^+} \phi(x)f(t,x)=
\lim_{x\to +\infty} \phi(x)f(t,x)=0 \text{ for all }t\in\R_+ 
\text{ and for all }
\phi\in 
{C^1(\R)}.
$$
{The previous assumption implies that} no analyst assigns a one-year target price for the S\&P 500 index equal to zero or infinity.

From \eqref{e:model}, we deduce formally that
\begin{equation}
 \label{e:model-weak}
 \begin{array}{l}
\displaystyle \frac{\mathrm{d}}{\dt}\int_{\R_+} f(t,x) \phi(x) \dx + \alpha \int_{\R_+}  \big((s(t) - x)f\big) \partial_x \phi(x) \dx  = 
\displaystyle \beta \int_{\R_+} f(t,x) \left [ \phi(\bar{x}) -\phi(x) \right] \dx
 \end{array}
\end{equation}
for all 
$\phi\in 
{C^1(\R)}$.

Denote
$$
m_0(t)= \int_{\R_+} f(t,x)  \dx \ \text{ and } \ m_1(t)= \int_{\R_+} x f(t,x)  \dx.
$$
When $\phi=1$, we deduce from the semi-weak form that
$$
\displaystyle \frac{\mathrm{d}}{\dt}m_0(t) =0,
$$
i.e.
$$
m_0= \int_{\R_+} f^{\text{in}}(x) \dx 
$$
for all $t\in\R_+$.

On the other hand, when $\phi=x$, the semi-weak form \eqref{e:model-weak} reduces to the following ordinary differential equation
governing the time evolution of $m_1$:
$$
\displaystyle \frac{\mathrm{d}}{\dt}m_1(t)  = 
 -q\beta m_1(t) + q\beta X(t)(1+\Delta)m_0.
$$
If we divide by the constant $m_0$, we deduce an ordinary differential equation which describes
the time evolution of the sentiment of the 
{analysts'} population:
$$
\displaystyle \frac{\mathrm{d}}{\dt}s(t)  = 
 q\beta \big [ X(t)(1+\Delta)-s(t)\big].
$$
This equation can be explicitly solved. Its solution is:
\begin{equation}
\label{e:s}
s(t) = \left(\int_{\mathbb{R}_+}x f^{\text{in}}(x) \mathrm{d}x\right)\left(\int_{\mathbb{R}_+}f^{\text{in}}(x) \mathrm{d}x\right)^{-1}
e^{-q\beta t} +
q \beta \int_0^t X(s)(1+\Delta)e^{q \beta(s-t)}  \mathrm{d}s.
\end{equation}
Hence, from now on, Equation \eqref{e:s} will replace Equation \eqref{e:s_old} within the initial-boundary value problem \eqref{e:model}.

\subsection*{Variance and risk implications}

A key advantage of the kinetic model is that it does not only provide the mean behavior of the system, but it also encapsulates richer distributional information. For instance, it allows us to compute higher-order moments such as the variance, as detailed below, which are crucial for understanding uncertainty and predicting risk or volatility.

In order to compute the variance, we begin with introducing the second-order moment
$$
m_2(t) = \int_{\mathbb{R}_+} x^2 f(t,x)\,\mathrm{d}x.
$$
By definition, the variance is given by
\begin{equation}\label{e:variance-def}
\mathcal{V}(t) = \frac{m_2(t)}{m_0} - \left( \frac{m_1(t)}{m_0} \right)^2,
\end{equation}
where $m_0$ and $m_1$ are the zeroth and first-order moments previously defined.

Applying the weak formulation \eqref{e:model-weak} with the test function $\phi(x)=x^2$, we obtain
$$
\frac{\mathrm{d}}{\mathrm{d}t} \int_{\mathbb{R}_+} x^2 f(t,x)\,\mathrm{d}x
+ \alpha \int_{\mathbb{R}_+} (s(t)-x) f(t,x)\,(2x)\,\mathrm{d}x
= \beta \int_{\mathbb{R}_+} f(t,x)\,\left(\bar{x}^2 - x^2\right)\,\mathrm{d}x,
$$
i.e.
$$
\frac{\mathrm{d}}{\mathrm{d}t} m_2(t)+
2\alpha \big[ s(t) m_1(t) - m_2(t) \big])
= \beta \left[(q^2 - 2q)\,m_2(t) + 2q(1-q)X(t)(1+\Delta)m_1(t)
    + q^2 X(t)^2(1+\Delta)^2 m_0\right],
$$
where we used that the quantities $q$, $X$ and $\Delta$ are independent of $x$.

Since $s(t)=m_1(t)/m_0$, we rewrite the previous equation as
\begin{equation}\label{e:m2-ODE}
\frac{\mathrm{d}}{\mathrm{d}t} m_2(t)
= \left[2\alpha + \beta(q^2-2q)\right] m_2(t)
- 2\alpha \frac{m_1(t)^2}{m_0} + 
2\beta {q} (1-q)X(t)(1+\Delta) m_1(t) 
+\beta q^2 X(t)^2(1+\Delta)^2 m_0.
\end{equation}

Combining \eqref{e:variance-def} and \eqref{e:m2-ODE}, using the ODE for $m_1$ and the definition of $s$, given in \eqref{e:s_old}, we deduce that $\mathcal{V}$ satisfies
\begin{equation}\label{e:variance-ODE}
\frac{\mathrm{d}}{\mathrm{d}t}\,\mathcal{V}(t)
= \left [2\alpha+\beta(q^2-2q)\right]\mathcal{V}(t)
+ \beta (q^2{-2q}) s(t)^2
+ 2\beta {q}
(1-2q)\,X(t)(1+\Delta) s(t)
+ \beta q^2\,X(t)^2(1+\Delta)^2
\end{equation}
with initial condition
$$
\mathcal{V}(0) = \frac{m_2^{\mathrm{in}}}{m_0} - \left( \frac{m_1^{\mathrm{in}}}{m_0} \right)^2,
$$
where
$$
m_1^{\mathrm{in}}= \int_{\R_+} x f^{\mathrm{in}}(x) \dx 
\text{\quad and\quad  }
m_2^{\mathrm{in}}= \int_{\R_+} x^2 f^{\mathrm{in}}(x) \dx. 
$$
The solution of \eqref{e:variance-ODE} is
\begin{equation}\label{e:variance-solution}
\mathcal{V}(t) = e^{\gamma t}\left[
\mathcal{V}(0) + \int_0^t e^{-\gamma \tau}\left(
\beta (q^2{-2q}) s(\tau)^2
+ 2\beta
{q}(1-2q)\,X(\tau)(1+\Delta) s(\tau)
+ \beta q^2\,X(\tau)^2(1+\Delta)^2
\right)\,\mathrm{d}\tau
\right],
\end{equation}
where $s$ is given by \eqref{e:s} and
$$
\gamma = 
2\alpha+\beta(q^2-2q).
$$
We immediately see that the sign of the coefficient $\gamma$ determines the intrinsic stability of the variance dynamics. In particular, when $\gamma<0$,
the linear term in \eqref{e:variance-ODE} acts as a damping factor. If the forcing term vanishes (i.e., the integral term), the variance decays exponentially.
This corresponds to a mean-reverting behavior where dispersion and risk tend to stabilize over time.
    
On the other hand, if $\gamma>0$, the regime is unstable. When $2\alpha>\beta q(2-q)$, the internal interaction dominates the consensus mechanism, and the variance grows exponentially even in the absence of external shocks. This scenario reflects an intrinsically unstable system, where disagreement amplifies over time.

Therefore, the kinetic model does not only provide the mean sentiment 
(i.e. the average analysts' recommendation) but also its distributional spread, which is a key forward-looking indicator of risk and volatility. 
In financial applications, such dispersion-based measures have been shown to anticipate periods of elevated uncertainty
\cite{volatility}, making $\mathcal{V}(t)$ a natural proxy for volatility forecasting and risk management.

\section{Existence and uniqueness of the solution} \label{exuniq}

By developing the transport term, our non-local partial differential equation can be written in the following form:
$$
\displaystyle \partial_t f  + \alpha  \left( s(t)  -  x\right ) \partial_x f =
 \frac \beta {1- q}  f\left (\frac x {1-q}  -\frac q{1-q} X(t)(1+\Delta)\right)\mathbf{1}_{\left (\frac x {1-q}  -\frac q{1-q} X(t)(1+\Delta)\geq 0 \right) }
 +(\alpha-\beta) f(t,x).
$$
Thanks to the method of characteristics, the previous equation can be written in integral form:
\begin{equation}
 \label{e:integral_form}
 \begin{array}{l}
\displaystyle
f(t,x)= f^{\text{in}}\left (xe^{\alpha t} - \alpha \int_0^t s(\theta) e^{\alpha \theta}\mathrm{d}\theta
\right)
e^{(\alpha-\beta)t} \mathbf{1}_{\left[\left (xe^{\alpha t} - \alpha \int_0^t s(\theta) e^{\alpha \theta}\mathrm{d}\theta\right)\geq 0\right]}+
\\[10pt]
\displaystyle
 \frac \beta {1- q} \int_0^t  e^{(\beta-\alpha)(\theta-t)}
 f\left (\frac x {1-q}  -\frac q{1-q} X(\theta)(1+\Delta)\right)\mathbf{1}_{\left (\frac x {1-q}  -\frac q{1-q} X(\theta)(1+\Delta)\geq 0\right)}
 \mathrm{d} \theta,
\end{array}
\end{equation}
with $s$ given by \eqref{e:s}.

Let $T \in\R_+^*$ and $C_b([0,T]\times\R_+)$ the space of uniformly bounded continuous functions on $[0,T]\times\R_+$,
such that
$$
\Vert f \Vert_{C_b([0,T]\times\R_+)} =\max_{(t,x)\in [0,T]\times\R_+} |f(t,x)| < +\infty.
$$
The following result holds:
\begin{Thm}
The integral problem \eqref{e:integral_form}, with initial condition $f^{\text{in}}\in C_b(\R_+)$, has one and only one solution $f\in C_b([0,T]\times\R_+)$.
\end{Thm}

\begin{proof}
Denote with
$$
F(f^{\text{in}})= f^{\text{in}}\left (xe^{\alpha t} - \alpha \int_0^t s(\theta) e^{\alpha \theta}\mathrm{d}\theta
\right)
e^{(\alpha-\beta)t} \mathbf{1}_{\left[\left (xe^{\alpha t} - \alpha \int_0^t s(\theta) e^{\alpha \theta}\mathrm{d}\theta\right)\geq 0\right]}
$$
and
with $\mathscr{T}\, : \, C_b([0,T]\times\R_+) \to C_b([0,T]\times\R_+)$ the operator such that
$$\mathscr{T}g=
\frac \beta {1- q} \int_0^t  e^{(\beta-\alpha)(\theta-t)}
 g\left (\frac x {1-q}  -\frac q{1-q} X(\theta)(1+\Delta)\right)\mathbf{1}_{\left (\frac x {1-q}  -\frac q{1-q} X(\theta)(1+\Delta)\geq 0\right)}
 \mathrm{d} \theta,
$$
where $g\in C_b([0,T]\times\R_+)$.

Hence, Equation \eqref{e:integral_form} can be written as a fixed-point problem:
\begin{equation}
\label{e:fixed_point}
f= F(f^{\text{in}})+\mathscr{T}f
\end{equation}
in $C_b([0,T]\times\R_+)$.
In what follows, we show that \eqref{e:fixed_point} has one and only one solution in $C_b([0,T]\times\R_+)$.

We first notice that
$$
\Vert F(f^{\text{in}})
\Vert_{C_b([0,T]\times\R_+)}  \leq 
\Vert  f^{\text{in}} \Vert_{C_b([0,T]\times\R_+)} e^{|\alpha-\beta| T}.
$$

On the other hand, the operator $\mathscr{T}$ is clearly linear and is bounded in $C_b([0,T]\times\R_+)$ with bounded norm 
$\Vert \mathscr{T} \Vert_{\mathcal{L}(C_b([0,T]\times\R_+))}$. Indeed,
$$
\Vert \mathscr{T}g \Vert_{C_b([0,T]\times\R_+)}  \leq 
\Vert g \Vert_{C_b([0,T]\times\R_+)} \left( \frac {\beta }{1- q} \right) \left(  \frac{1- e^{-(\beta-\alpha) T}}{\beta-\alpha} \right)
$$
for all $g\in C_b([0,T]\times\R_+)$.

Hence,
$$
\Vert \mathscr{T} \Vert_{\mathcal{L}(C_b([0,T]\times\R_+))} = \sup_{\Vert g \Vert_{C_b([0,T]\times\R_+)} =1} 
\Vert \mathscr{T}g \Vert_{C_b([0,T]\times\R_+)}  \leq 
\left( \frac {\beta }{1- q} \right) \left(  \frac{1- e^{-(\beta-\alpha) T}}{\beta-\alpha} \right) .
$$

The next step consists in evaluating the norm of $\mathscr{T}^n$, for all $n\in \mathbb{N}^*$.
By estimating the integral in a slightly different way, we have that
\begin{equation}
\label{e:estimate}
\begin{array}{c}
\displaystyle  \Vert \mathscr{T}^n g \Vert_{C_b([0,T]\times\R_+)}  \leq \left( \frac {\beta }{1- q} \right)^n 
\Vert g \Vert_{C_b([0,T]\times\R_+)} 
\int_0^t\int_0^{t_1}\cdots\int_0^{t_{n-1} }\mathrm{d}t_1 \mathrm{d}t_1 \cdots \mathrm{d}t_n \leq \\[10pt]
\displaystyle  \left( \frac {\beta }{1- q} \right)^n 
\Vert g \Vert_{C_b([0,T]\times\R_+)} \frac{T^n}{n!}.
\end{array}
\end{equation}
We now claim that
$$
f(t,x)=\sum_{n=0}^{+\infty} \mathscr{T}^n F(f^{\text{in}})
$$
is a well defined function of $C_b([0,T]\times\R_+)$ and that it solves the fixed-point problem \eqref{e:fixed_point}.
We indeed have that
\begin{equation*}
\begin{array}{c}
\displaystyle \Vert f \Vert_{C_b([0,T]\times\R_+)} \leq \displaystyle \sum_{n=0}^{+\infty} 
\Vert \mathscr{T} ^n \Vert_{\mathcal{L}(C_b([0,T]\times\R_+))}
\Vert F(f^{\text{in}}) \Vert_{C_b([0,T]\times\R_+)} \leq \\[10pt]
\displaystyle  \sum_{n=0}^{+\infty} 
\Vert  f^{\text{in}} \Vert_{C_b([0,T]\times\R_+)} e^{|\alpha-\beta| T}
\left( \frac {\beta T}{1- q} \right)^n 
\frac1 {n!}=
\Vert  f^{\text{in}} \Vert_{C_b([0,T]\times\R_+)} e^{|\alpha-\beta| T}\exp \left( \frac {\beta}{1- q} T\right).\\
\end{array}
\end{equation*}
Hence the series is normally convergent with respect to the natural norm of $C_b([0,T]\times\R_+)$. Consequently, $f$ belongs to $C_b([0,T]\times\R_+)$. 

Moreover, $f$ solves \eqref{e:fixed_point} because
$$
f(t,x)=\sum_{n=0}^{+\infty} \mathscr{T}^n F(f^{\text{in}})=  F(f^{\text{in}}) + \sum_{n=1}^{+\infty} \mathscr{T}^n F(f^{\text{in}})=
F(f^{\text{in}}) +  \mathscr{T} \sum_{n=0}^{+\infty} \mathscr{T}^n F(f^{\text{in}}) = F(f^{\text{in}}) +  \mathscr{T} f.
$$
Uniqueness follows by contradiction. Suppose that there exist two distinct solutions $f_1$ and $f_2$ of \eqref{e:fixed_point}.

Hence, their difference $\tilde f = f_1-f_2$ has finite and non-zero norm $\Vert \tilde f \Vert_{C_b([0,T]\times\R_+)}$. By linearity,  $\tilde f$ solves
$$
\tilde f= F(f^{\text{in}}) + \mathscr{T}\tilde  f,
\qquad f^{\text{in}}=0.
$$
By iteration we have that
$$
\tilde f=  \mathscr{T}\tilde  f = \dots = \mathscr{T}^n\tilde f
$$
for all $n\in\mathbb{N}^*$ and, by passing to the norm, we have
$$
 \Vert \tilde f \Vert_{C_b([0,T]\times\R_+)}=  
 \Vert \mathscr{T}^n\tilde  f \Vert_{C_b([0,T]\times\R_+)}
$$
for all $n\in \mathbb{N}^*$. Hence,
$$
 \Vert \tilde f \Vert_{C_b([0,T]\times\R_+)}= 0 \text{ or } 
 \Vert \mathscr{T}^n\tilde  f \Vert_{C_b([0,T]\times\R_+)}.
$$
The estimate given in \eqref{e:estimate} guarantees that $ \Vert \mathscr{T^n}\tilde  f \Vert_{C_b([0,T]\times\R_+)}
\to 0 $ when $n\to +\infty$, and hence this contradict the fact that $f_1$ and $f_2$ are distinct solutions of the fixed-point problem
\eqref{e:fixed_point}.
\end{proof}

\section{Numerical tests} \label{s:numerical}

In this section we provide some numerical simulations which describe the behaviour of the model and compare it with the actual data.

\subsection{Time evolution of the price sentiment} \label{ss:sentiment}

We study the problem
\begin{equation}
\label{e:test}
\left\{
\begin{array}{l}
\displaystyle \frac{\mathrm{d}}{\dt}s(t)  = 
 q\beta \big [ X(t)(1+\Delta)-s(t)\big]\\[10pt]
\displaystyle s(0)= s^{\text{in}} =  \left(\int_{\mathbb{R}_+}x f^{\text{in}}(x) \mathrm{d}x\right)\\
\end{array}
\right.
\end{equation}
with given
$\Delta\in\R$, which has the explicit solution
$$
\displaystyle s(t) =  \left(\int_{\mathbb{R}_+}x f^{\text{in}}(x) \mathrm{d}x\right)\left(\int_{\mathbb{R}_+}f^{\text{in}}(x) \mathrm{d}x\right)^{-1}\ e^{-q\beta t} +
q \beta \int_0^t X(\theta)(1+\Delta)e^{q \beta(\theta-t)}  \mathrm{d} \theta
$$
We have developed a \texttt{MATLAB} script that implements a numerical solution for the differential equation \eqref{e:test} using a standard fourth-order Runge-Kutta method. We recall that the spot price \( X(t) \) is derived from time-series data of the S\&P 500 index, covering the period from May 26th, 2009, to  {June 2nd, 2025}.

To calibrate the model parameters and assess its accuracy, we divided the dataset into two subsets. The first subset, referred to as the training set, includes S\&P 500 index data from May 26th, 2009, to December 31st, 2016. For the sentiment data $s_{\text{meas}}(t)$, we used the average of one-year analysts' forecasts for the S\&P 500 index over the same period. The training set consists of 1,987 data points.

These training data were used to optimize the model's free parameters ($q$, $\beta$ and $\Delta$) using a combination of global and local optimization techniques. Specifically, a global optimizer was employed to explore the parameter space broadly, followed by a local optimizer for fine-tuning. The distance metric used in the optimization process was the standard Euclidean distance between the forecast vector $s(t)$ and the measured time series
 $s_{\text{meas}}(t)$. The optimized parameters were determined as $q = 0.28$, $\beta = 6.05$ and $\Delta = 0.143$, which were subsequently used for forecasting.

It is important to note that, since 2009, the S\&P 500 index has experienced significant growth interspersed with periods of volatility. Our dataset begins at the point when the index started recovering from the 2008 financial crisis. By 2013, the index had surpassed its pre-crisis levels and continued to grow steadily until 2019, despite challenges such as trade tensions and global economic uncertainties.

In 2020, the COVID-19 pandemic caused a sharp decline, with the index losing approximately 34\% of its value by March. However, it recovered quickly due to government stimulus and strong performance in technology and healthcare sectors. In 2022, inflationary pressures and geopolitical tensions, including the war in Ukraine, triggered market corrections.

As a result, our training set is based on data from a period of market growth and does not include data from more volatile periods.

Using the optimized parameters, we simulated the period from January 3rd, 2017, to {June 2nd, 2025}, and compared the results with the corresponding time series, consisting of {2,193} data points.

Figure \ref{fig:2} presents the results of our model compared to the average of one-year analysts' forecasts for the forecasting set, which spans from January 3rd, 2017, to {June 2nd, 2025}, and to the cointegration model equilibrium which parameters are estimated over the same period as the kinetic model. For this simulation, the sentiment $s(0)$ was initialized using the value computed on January 3, 2017, $s(0) = 2463.93$.  It is clear that the kinetic model better captures the long-term dynamics compared to the cointegration model. Moreover, while the long-term equilibrium estimated by the cointegration model is much more volatile, closely mirroring the fluctuations of the S$\&$P 500 (due to $\alpha$ being near 1), the kinetic model does not suffer from the same volatility issues.

\begin{figure}[h!]
\begin{center}
\includegraphics[height=9cm]{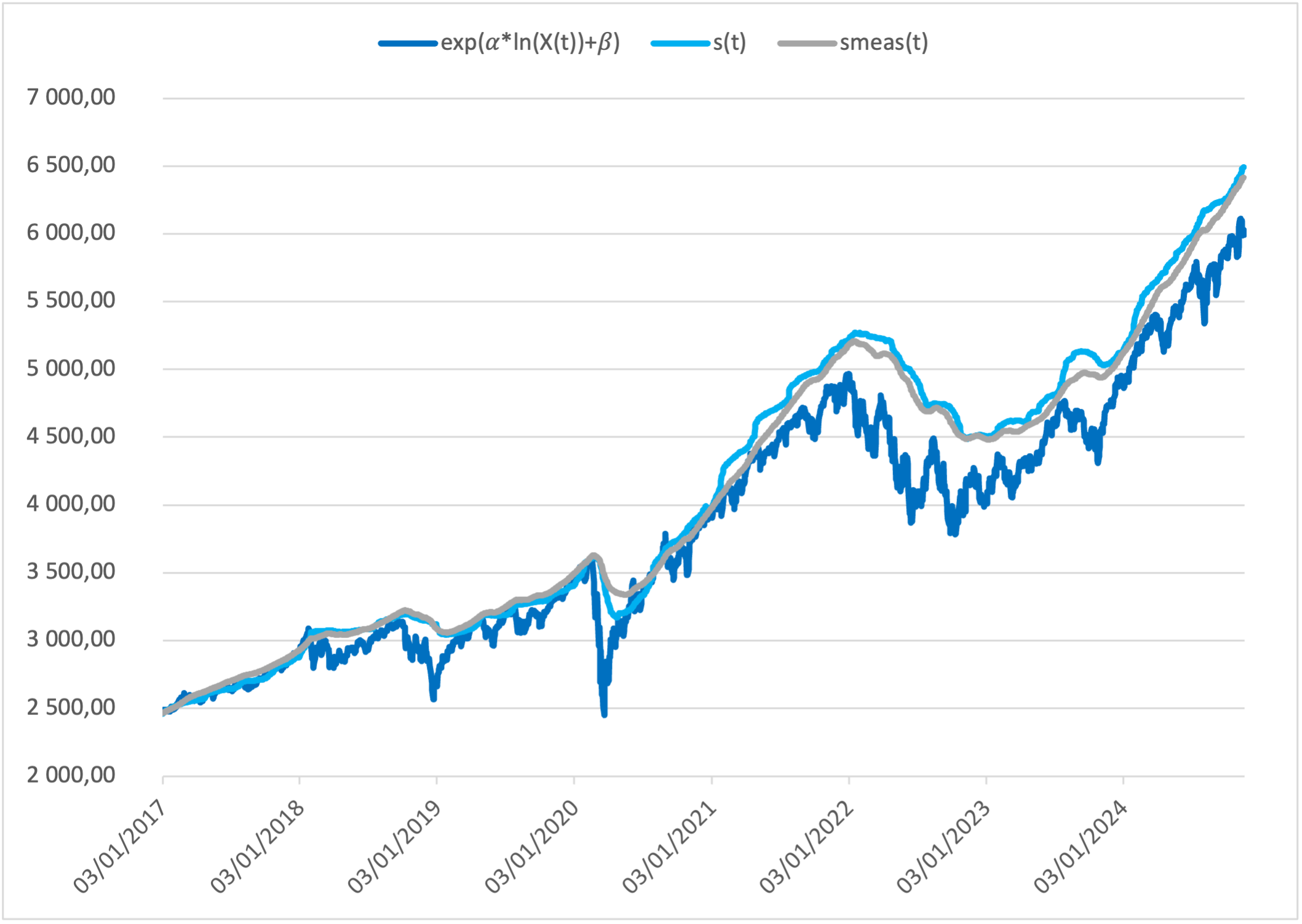}
\caption{ 
Forecasts $ s(t) $  of the kinetic model {(black line)}, Forecasts $ \exp{(\alpha\ln (X(t))+ \beta)}$  of the cointegration model {(grey line)} and the time series of the average of one-year analysts' forecasts for the S\&P 500 index $ s_{\text{meas}}(t)$  {(light grey line)}, from January 3rd, 2017, to  {June 2nd, 2025}.} 
\label{fig:2}
\end{center}
\end{figure}

Finally, in Figure \ref{fig:3}, we plot the relative error between the forecasts generated by our model and the real data for the average of one-year analysts' forecasts.

\begin{figure}
\begin{center}
\includegraphics[height=9cm]{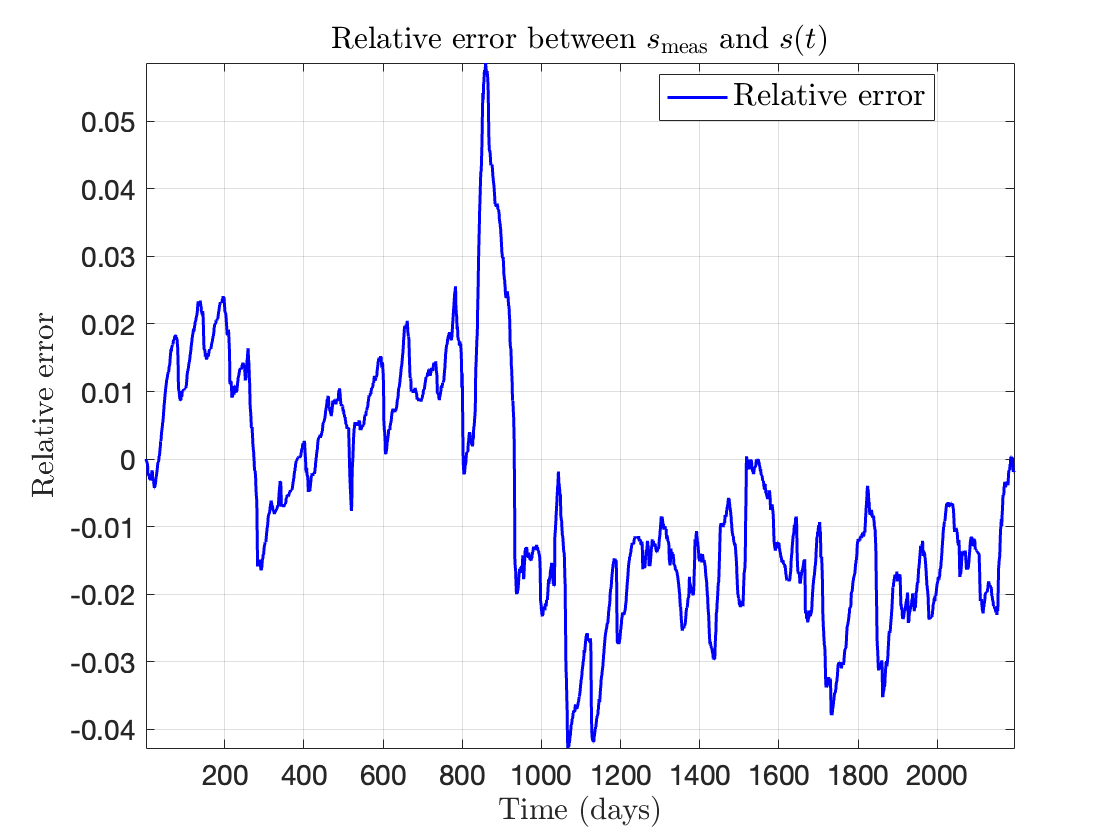}
\caption{Analysis of the relative error between the forecasts $s(t)$ of the model and the time series of the average of one-year analysts' forecasts for the S\&P 500 index $s_{\text{meas}}$, from January 3rd, 2017, to {June 2nd, 2025.}} 
\label{fig:3}
\end{center}
\end{figure}

The statistical analysis of the relative errors reveals several insights into the robustness of the forecast. The mean relative error is equal to {0,015485}, while the median is equal to {0,014096}, indicating a slight systematic {over}estimation. This bias can be the consequence of the different market trends of the training set with respect to the forecasting set. The standard error, {0,000213}, reflects a low level of dispersion, showing that the predictions are consistently close to the actual values. 

The time series visualization of the relative errors shows how the errors fluctuate over time, with most values remaining close to zero. This consistency highlights the robustness of the forecast within a time horizon of approximately seven years and suggests that the model performs reliably across different periods.

One such event is the March 2020 market crash, triggered by the COVID-19 pandemic, which was followed by a rapid and sustained rally due to fiscal stimulus, monetary easing, and large-scale asset purchases by the Federal Reserve. Our model tended to overestimate forecasts before the crash and underestimate them afterward, reflecting a shift in market dynamics and a moderate increase in analysts' overall optimism post-crash.

The model's parameters were originally estimated using data up to 2016. To improve forecasting accuracy, a more dynamic calibration approach could be adopted, such as periodic updates of the parameters incorporating the entire historical dataset or a rolling window, for example of the past five years, to recalibrate at each time step.

Future research will further explore these methodologies, extending the analysis to other market indices to assess the generalisability of the findings.

\section{Concluding remarks}

Our study indicates that the analysts' consensus for the one-year S\&P 500 price target can be forecast by assuming it is primarily driven by the current index level, supplemented by a premium 
(of approximately 10$\%$).
While some individual analysts may possess genuine predictive power, this is clearly not the case for the 
{aggregated one-year price forecasts.}

This initial intuition was informed by the VECM analysis, but it is the kinetic model that allows for a precise description of how analysts weigh the consensus opinion relative to the current price. Crucially, it appears analysts place greater weight on the market price than on the opinions of their peers.

This comparative analysis decisively establishes the kinetic model's superiority over the traditional VECM for this specific forecasting task. The VECM is severely limited by the necessity of restrictive assumptions like normality and homoscedasticity -- assumptions frequently violated in volatile equity data -- which cause it to suffer from significant forecasting deviations.

In sharp contrast, the kinetic model is inherently more robust: its design specifically targets out-of-equilibrium phenomena, requiring no underlying distributional hypotheses. Consequently, the kinetic model provides highly stable and convincing forecasts, maintaining its efficacy even during chaotic and unpredictable market periods and over a long-term horizon.

\section*{Use of AI tools declaration}
The English in this paper has been refined using Artificial Intelligence tools.

\section*{Acknowledgments}
FS acknowledges the support of INdAM, GNFM group.

\section*{Conflict of interest}

The authors state that there are no conflicts of interest nor ethical issues.


\end{document}